\documentclass[11pt,a4paper]{article} 

\usepackage{amsmath} 
\usepackage{amssymb} 
\usepackage{mathrsfs} 
\usepackage{hyperref} 
\numberwithin{equation}{section} 
\usepackage{cite} 

\evensidemargin=\oddsidemargin
\newtheorem{theorem}{Theorem}[section] 
\newtheorem{lemma}[theorem]{Lemma} 
\newtheorem{corollary}{Corollary}[section] 
\newtheorem{remark}{Remark}[section] 

\newenvironment{proof}[1][Proof]{\begin{trivlist}
\item[\hskip \labelsep {\bfseries #1}]}{\end{trivlist}}

\numberwithin{equation}{section} 
\allowdisplaybreaks              

\begin{document}


\title {Generalized Hilbert Operator Acting on Hardy Spaces
\footnote{   The research
was supported by Zhejiang Province Natural Science Foundation of China (Grant No. LY23A010003).}}
\author{  Huiling Chen\footnote{E-mail address:  HuillingChen@163.com}\quad\quad Shanli Ye\footnote{Corresponding author, E-mail address: slye@zust.edu.cn} \\
(\small \it School of Science, Zhejiang University of Science and Technology,
Hangzhou 310023, China)}
 \date{}
\maketitle
\begin{abstract} Let  $\alpha>0$ and $\mu$ be a positive Borel measure on the interval $[0,1)$. The Hankel matrix $\mathcal{H}_{\mu,\alpha}=(\mu_{n,k,\alpha})_{n,k\ge0}$ with entries $\mu_{n,k,\alpha}=\int_{[0,1)}^{}\frac{\Gamma(n+\alpha)}{\Gamma(n+1)\Gamma(\alpha)}t^{n+k}d\mu(t)$, induces, formally, the generalized-Hilbert operator as
$$
\mathcal{H}_{\mu,\alpha}\left ( f \right ) \left ( z \right ) =\sum_{n=0}^{\infty} \left (\sum_{k=0}^{\infty} \mu_{n,k,\alpha}a_k  \right )z^n,z\in\mathbb{D}
$$
where $f(z)={\textstyle \sum_{k=0}^{\infty }} a_kz^k$ is an analytic function in $\mathbb{D}$.
This article is devoted to study the measures $\mu$ for which $\mathcal{H}_{\mu,\alpha }$ is a bounded(resp., compact) operator from $H^p(0<p\le1)$ into $H^p(1\le q<\infty)$. Then, we also study the analogous problem in the Hardy spaces $H^p(1\le p\le2)$. Finally, we obtain the essential norm of $\mathcal{H}_{\mu,\alpha}$ from $H^p(0<p\le1)$ into $H^p(1\le q<\infty)$.
  \\
{\small\bf Keywords}\quad
Hilbert operator; Hardy space; Carleson measure; Essential norm
 \\
    {\small\bf 2020 MR Subject Classification }\quad 47B38, 47B35, 30H10 \\

\end{abstract}

\maketitle 
\section{Introduction}

\hspace{1.2em}
\quad   
\qquad   
\par  
Define the open unit disk $\mathbb{D}$ in the complex plane $\mathbb{C}$ as $\mathbb{D}=\left \{ z\in \mathbb{C}:\left | z \right |<1  \right \}$, and let $H(\mathbb{D})$ represent the set of all holomorphic functions in $\mathbb{D}$.

 If $0<r<1$ and $f\in H(\mathbb{D})$, we set
\begin{align}
   & M_p (r,f)=\left( \frac{1}{2\pi} \int_0^{2\pi} |f(re^{i \theta})|^p d\theta \right)^\frac{1}{p}, \quad 0<p<\infty.  \notag\\
   & M_\infty (r,f)=\sup_{|z|=r}|f(z)|.\notag
\end{align}

For $0<p\leq \infty$, the Hardy space $H^p$ consists of those $f \in H(\mathbb{D})$ with
$$||f||_{H^p} \overset{def}{=} \sup_{0<r<1}M_p(r,f)<\infty.$$

We refer to \cite{3} for the terminology and findings on Hardy spaces.

The space $BMOA$ consists of those functions $f \in H^1$ whose boundary values has bounded mean oscillation on $\partial \mathbb{D}$, in accordance with the definition by John and Nirenberg. Numerous properties and descriptions can be attributed to $BMOA$ functions. Let us mention the following: for $a \in \mathbb{D}$, let $\varphi_a$ be the M$\ddot{o}$bius transformation
defined by $\varphi_a(z)=\frac{a-z}{1-\overline{a}z}$. If $f$ is an analytic function in $\mathbb{D}$, then $f \in BMOA $ if and only if
$$||f||_{BMOA} \overset{def}{=} |f(0)|+||f||_*< \infty,$$
where
$$||f||_* \overset{def}{=}\sup_{a\in\mathbb{D}}\left\{\int_{\mathbb{D}}|f'(z)|^2(1-|\varphi_a(z)|^2)dA(z)\right\}^{1/2},$$
 where $dA(z)=\frac{1}{\pi}dxdy$ denotes the normalized Lebesgue area measure on $\mathbb{D}$. For an exposition on the theory of $BMOA$ functions, one should review the content in reference \cite{4}.

The Bloch space $\mathcal{B}$ consists of those functions $f\in H(\mathbb{D})$ with
$$\|f\|_{\mathcal{B}}= |f(0)|+\sup_{z\in\mathbb{D}}(1-|z|^2)|{f}'(z)|<\infty.$$

Consult references \cite{1,2} for the terminology and findings concerning Bloch-type spaces. It is a recognized fact that $BMOA\varsubsetneq\mathcal{B}$.

For $f\left ( z \right ) = {\textstyle \sum_{n=0}^{\infty}} a_nz^n\in H\left ( \mathbb{D}  \right )$ and any complex parameters $\beta$ and $\gamma$ such that neither $1+\beta$ nor $1+\beta+\gamma$ is a negative integer, $R^{\beta,\gamma}$ called the fractional differential operator\cite{19} as follows:
$$
R^{\beta,\gamma}f(z)=\sum_{n=0}^{\infty}\frac{\Gamma(2+\beta)\Gamma(n+2+\beta+\gamma)}{\Gamma(2+\beta+\gamma)\Gamma(n+2+\beta)}a_nz^n.
$$
\par
Let $\mu$ is a finite positive Borel measure on $[0,1)$. The Hankel matrix defined by its elements $\mu_{n,k}=\mu_{n+k}$ for $n,k\ge0$, where $\mu_{n}=\int_{[0,1)}t^nd\mu(t)$, formally represents the Hilbert operator
$$\mathcal{H}_\mu(f)(z)=\sum_{n=0}^\infty\left(\sum_{k=0}^\infty \mu_{n,k}a_k\right)z^n ,  ~z\in \mathbb{D},$$
where $f(z)=\sum_{n=0}^\infty a_nz^n$ is an analytic function in $\mathbb{D}$. Similarly, if $\mu$ is a finite positive Borel measure on $[0,1)$ and $\alpha>0$, we use $\mathcal{H}_{\mu,\alpha}=(\mu_{n,k,\alpha})_{n,k\ge0}$ to denote the Hankel matrix $\left ( \mu_{n,k} \right ) _{n,k\ge 0}$ with entries $\mu_{n,k,\alpha}= \int_{[0,1)}^{}\frac{\Gamma(n+\alpha)}{\Gamma(n+1)\Gamma(\alpha)}t^{n+k}d\mu(t)$. The matrix $\mathcal{H}_{\mu,\alpha}$ can be regarded as an operator on $H(\mathbb{D})$ by its action on the Taylor coefficients:
$$
a_n\to\sum_{k=0}^{\infty}\mu_{n,k,\alpha}a_k,\quad n=0,1,2,\cdots.
$$
That is, for $f\left ( z \right ) = {\textstyle \sum_{n=0}^{\infty}} a_nz^n\in H\left ( \mathbb{D}  \right )$, the generalized Hilbert operator defined as follows:
\begin{align}\label{eqn1.1}
\mathcal{H}_{\mu,\alpha}(f)(z)=\sum_{n=0}^{\infty}\left ( \sum_{k=0}^{\infty}\mu_{n,k,\alpha}a_{k}  \right ) z^n,z\in\mathbb{D},
\end{align}
whenever the right hand side is well defined and defines a function in $H(\mathbb{D})$. The operator $\mathcal{H}_{\mu,\alpha}$ is also called fractional derivative Hilbert operator since it is easy to see that $R^{-1,\alpha}\mathcal{H}_\mu(f)=\mathcal{H}_{\mu,\alpha}(f)$.

The generalized Hilbert operator $\mathcal{H}_{\mu,1}$ has been methodically studied in many different spaces, such as Bergman spaces, Bloch spaces, Hardy spaces(e.g.\cite{7,8,9,10,11,12}). For the case $\alpha=2$, $\mathcal{H}_{\mu,2 }$ called the Derivative-Hilbert operator which has been studied in \cite{xu-ye1,xu-ye2,ye-xu,17,18,15}. In \cite{19, 22, ye-feng, 16}, the operator $\mathcal{H}_{\mu,\alpha}$ were
called the generalized Hilbert operators.

In addition, the operator $\mathcal{H}_{\mu,\alpha }$ is related to the generalized integral-Hilbert operator $\mathcal{I}_{{\mu}_\alpha }(\alpha>0)$ defined by
\begin{align}\label{eqn1.2}
 \mathcal{I}_{\mu,\alpha }\left ( f \right ) \left ( z \right ) =\int_{[0,1)}^{} \frac{f\left ( t \right ) }{\left ( 1-tz \right )^\alpha  }d\mu \left ( t \right ),
\end{align}
whenever the right hind side makes sense and defines an analytic function in $\mathbb{D}$. If $\alpha=1$, then $\mathcal{I}_{{\mu}_\alpha }$ is the integral operator $\mathcal{I}_{{\mu}}$. After that, Ye and Zhou characterized the measures $\mu$ for which $\mathcal{H}_{\mu,2 }= \mathcal{I}_{\mu,2 }$ are bounded (resp., compact) on the Bloch space \cite{15} and on the Bergman spaces \cite{18}. In this article, we can also gain the operators $\mathcal{H}_{\mu,\alpha }$ and $\mathcal{I}_{\mu,\alpha }$ are intricately connected for $\alpha>0$.

Let us review the concept of the Carleson-type measures, which is a useful tool for understanding Banach spaces of analytic functions.

If $I\subset\partial \mathbb{D}$ in an arc, $|I|$ denotes the length of $I$, the Carleson square $S(I)$ is defined as
$$
 S(I)=\left\{z=re^{it}:e^{it}\in I, 1-\frac{|I|}{2\pi}\leq r < 1 \right\}.
$$

Suppose that $\mu$ is a positive Borel measure on $\mathbb{D}$. For $0\leq \beta < \infty$ and $ 0<s< \infty $, we say that $\mu$ is a $\beta$-logarithmic $s$-Carleson measure if there exists a positive constant $C$ such that
 $$\sup_I\frac{\mu(S(I))(\log\frac{ 2\pi }{|I|})^\beta}{|I|^s} \leq C, \quad  \quad I \subset \partial \mathbb{D}.$$

If $\mu(S(I))(\log\frac{ 2\pi }{|I|})^\beta=o(|I|^s)$ as  $|I|\rightarrow 0$, we say that $\mu$ is a vanishing $\beta$-logarithmic $s$-Carleson measure.

  A positive Borel measure on $[0, 1)$  can also be seen as a Borel measure on $\mathbb{D}$ by identifying it with the measure $\mu$ defined by
$$\tilde{\mu}(E)=\mu(E\bigcap [0,1)),$$
for any Borel subset $E$ of $\mathbb{D}$.  Then we say that $\mu$ is a $\beta$-logarithmic $s$-Carleson measure if there exists a positive constant $C$ such that
$$
\mu ([t,1)) \log^\beta\frac{e}{1-t } \le C(1-t)^s,\quad for\ all \ 0\le t<1.
$$

In detail, $\mu$ is a $s$-Carleson measure if $\beta=0$. If $\mu$ satisfies
$$
\lim_{t\to1^-}\frac{\mu ([t,1)) \log^\beta\frac{e}{1-t }}{(1-t)^s} =0,
$$
we say that $\mu$ is a vanishing $\beta$-logarithmic $s$-Carleson measure(see \cite{5,6}).

The essential norm of a continuous linear operator T between two Banach spaces P and Q is defined as the distance from T to the set of compact operators
K. Mathematically, this is expressed as: $\|T\|_{e,P\to Q}=inf\left\{\|T-K\|_{P\to Q}:K \ is\ compact\right\}$, where $\|\cdot\|$ is the operator norm. It is easy to see that $\|T\|_{e,P\to Q}=0$ if and only if T itself is a compact operator. This concept is further studied in references \cite{24,25}.

In this paper, we focus on identifying the positive Borel measure $\mu$ such that $\mathcal{H}_{\mu,\alpha}(\alpha>0)$ is bounded (resp. compact) from $H^p(0<p\le 1)$ into $H^p(1\le q<\infty)$. Additionally, we also do similar work for the operators acting on $H^p(1\le p\le2)$. Subsequently, we determine the Hilbert-Schmidt class on $H^2$ for all $\alpha>0$. Ultimately, we identify the essential norm of $\mathcal{H}_{\mu,\alpha}$ from $H^p(0<p\le1)$ into $H^p(1\le q<\infty)$.

Throughout this work,the symbol $C$ represents an absolute constant that depends only on the parameters specified in parentheses, though it may vary between different instances. We employ the notation "$J\lesssim K$" if there exists a constant $C=C(\cdot)$ such that $J\le CK$ and $J\gtrsim K$ is interpreted in a comparable fashion. For any given $q>1$, ${q}'$ is used to denote the conjugate index of $q$, which satisfies the equation $1/q+1/{q}'=1$.

\section{Preliminary results}
\begin{lemma}\label{Th1.1}\cite{22}
Suppose that  $0<p<\infty$ and $\alpha>0$, let $\mu$ be a positive Borel measure on $[0,1)$. Then for every $f\in H^p$, (\ref{eqn1.1}) is a defined analytic function in $\mathbb{D}$ in either of the two following cases:

(i) If $0<p\le1$, $\mu$ is a $\frac{1}{p}$-Carleson measure;

(ii) If $1<p<\infty$, $\mu$ is a $1$-Carleson measure. 

Moreover, in such as cases we obtain that
$$
\mathcal{H}_{\mu,\alpha}(f)(z)=\int_{[0,1)}^{}\frac{f(t)}{(1-tz)^\alpha}d\mu(t)=\mathcal{I}_{\mu,\alpha}(f)(z).
$$
\end{lemma}  

\begin{lemma}\label{Lm1.2}\cite{13} 
Let $\gamma>0$ and $f\in H(\mathbb{D})$. If there exists a real parameter $\beta$ such that neither $1+\beta$ nor $1+\beta+\gamma$ is a negative integer, then the following statements are equivalent.

(i)$f\in\mathcal{B};$

(ii) The function $(1-|z|^2)^\gamma R^{\beta,\gamma}f(z)$ is bounded in $\mathbb{D}$.

Furthermore,
$$
\|f\|_\mathcal{B}\asymp|f(0)|+\sup\limits_{z\in\mathbb{D}}(1-|z|^2)^\gamma\left|R^{\beta,\gamma}f(z)\right|.
$$
\end{lemma}
\begin{lemma}\label{Lm1.3}\cite{23}
Let $\tau$ be real. Then the integral
$$
G(a)=\int_{0}^{2\pi}\frac{d\theta}{\left|1-\bar{a} e^{i\theta}\right|^{1+\tau}}
$$
have the following properties.

(i) If $\tau>0$, then $G(a)\asymp\frac{1}{(1-|a|^2)^\tau};$

(ii) If $\tau=0$, then $G(a)\asymp\log\frac{e}{1-|a|^2};$

(iii) If $\tau<0$, then $G(a)\asymp1.$
\end{lemma}

\begin{lemma}\label{lm2.5}
Let $0<\alpha<\infty$ and $\mathcal{H}_{\mu,\alpha }$ is a bounded operator from $H^p$ into $H^q$. Then $\mathcal{H}_{\mu,\alpha }$ is a compact operator from $H^p$ into $H^q$ if and only if for any bounded sequence $\left \{ f_n \right \} $ in $H^p$ which converges to $0$ uniformly on every compact subset of $\mathbb{D}$, we have $\lim_{n\to\infty}\|\mathcal{H}_{\mu,\alpha }\|_{H^q}=0$.
\end{lemma}
\par
The argument can be proved just as in \cite[Proposition 3.11]{21}, we omit the details.

\begin{lemma}\label{lm2.6}\cite[Theorem 6.4]{3}
If $g(z)=\sum_{n=0}^{\infty}b_nz^n\in H^p,0<p\le1$, then
$$b_n=o(n^{1/p-1})$$
and
$$|b_n|\lesssim n^{1/p-1}\|g\|_{H^p}.$$
\end{lemma}

\begin{lemma}\label{lm2.7}\cite[Theorem 6.2]{3}
If $g(z)=\sum_{n=0}^{\infty}b_nz^n\in H^p,0<p\le2$, then

$$\left\{\sum_{n=0}^{\infty}(n+1)^{p-2}|b_n|^p\right\}^{1/p}\lesssim\|g\|_{H^p}.$$

\end{lemma}

\section{Boundedness of $\mathcal{H}_{\mu,\alpha }$ acting on $H^p$}

In this section, we qualify those measures $\mu$ for which $\mathcal{H}_{\mu,\alpha}$ is a bounded operator from $H^p$ to $H^q$.

\begin{theorem}\label{Th1.2}  
Suppose that $0<p\le1$ and $\alpha>0$. Let $\mu$ be a positive Borel measure on $[0,1)$ which satisfies the condition in Lemma \ref{Th1.1}.

(i) If $\alpha>1$, then $\mathcal{H}_{\mu,\alpha}:H^p\rightarrow H^1$ is bounded if and only if $\mu$ is a $\left(\frac{1}{p}+\alpha-1\right)$-Carleson measure;

(ii) If $0<\alpha<1$, then $\mathcal{H}_{\mu,\alpha}:H^p\rightarrow H^1$ is bounded if and only if $\mu$ is a $\frac{1}{p}$-Carleson measure.
\end{theorem}
\begin{remark}
 From  \cite[Theorem 1.2]{9}, we know that $\mathcal{H}_{\mu,1}:H^1\rightarrow H^1$ is bounded if and only if $\mu$ is a 1-logarithmic 1-Carleson measure.
\end{remark}

\begin{proof}
Since $\mu$ satisfies the condition in Lemma \ref{Th1.1}, it is easy to see that
\begin{align}\label{eqn3.1}
\int_{[0,1)}^{}|f(t)|d\mu(t)\lesssim\|f\|_{H^p},\quad for \ all \ f\in H^p.
\end{align}
For $0\le r<1$, $f\in H^p$ and $g\in H^1$, we have that
$$
\begin{aligned}
&\int_{0}^{2\pi}\int_{[0,1)}^{}\left|\frac{f(t)g(re^{i\theta})}{(1-rte^{-i\theta})^\alpha}\right|d\mu(t)d\theta\\
\le&\frac{1}{(1-r)^\alpha}\int_{[0,1)}^{}|f(t)|d\mu(t)\int_{0}^{2\pi}|g(re^{i\theta})|d\theta\\
\lesssim&\frac{\|f\|_{H^p}\|g_r\|_{H^1}}{(1-r)^\alpha}\lesssim\frac{\|f\|_{H^p}\|g\|_{H^1}}{(1-r)^\alpha}<\infty.
\end{aligned}
$$
where, $g_r$ is a function defined by $g_r(z)=g(rz),z\in\mathbb{D}$. Let $g(z)=\sum_{n=0}^{\infty}b_nz^n$. Using Fubini's theorem and a simple calculation, we deduce that
\begin{align}\label{eqn3.2}
&\frac{1}{2\pi}\int_{0}^{2\pi}\overline{\mathcal{H}_{\mu,\alpha}(f)(re^{i\theta})}g(re^{i\theta})d\theta\notag\\
=&\int_{[0,1)}^{}\sum_{n=0}^{\infty}\frac{\Gamma(n+\alpha)}{\Gamma(n+1)\Gamma(\alpha)}b_n(r^2t)^n\overline{f(t)}d\mu(t)\notag\\
=&\int_{[0,1)}^{}R^{-1,\alpha-1}g(r^2t)\overline{f(t)}d\mu(t).
\end{align}
(i) Recall that the Fefferman's duality theorem (see \cite{4}) shows that $(H^1)^\ast\cong BMOA$ and $(VMOA)^\ast\cong H^1$ under the pairing
$$
<F,G>=\lim_{r\to1}\frac{1}{2\pi}\int_{0}^{2\pi}\overline{F\left(re^{i\theta}\right)}G\left(e^{i\theta}\right)d\theta, \quad F\in H^1,\ G\in BMOA (resp.,VMOA).
$$
Therefore, it follows from (\ref{eqn3.2}) that $\mathcal{H}_{\mu,\alpha}:H^p\rightarrow H^1$ is bounded if and only if
\begin{align}\label{eqn-3.5}
\left|\int_{[0,1)}^{}R^{-1,\alpha-1}g(r^2t)\overline{f(t)}d\mu(t)\right|\lesssim\|f\|_{H^p}\|\|g\|_{BMOA},\quad f\in H^p,g\in VMOA.
\end{align}

Assume that $\mathcal{H}_{\mu,\alpha}:H^p\rightarrow H^1$ is bounded, take the families of text functions
\begin{align}\label{eqn-abc}
f_a(z)=\frac{(1-a^2)^{\frac{1}{p}}}{(1-az)^{\frac{2}{p}}}\quad and \quad g_a(z)=\log\frac{e}{1-az},\quad \frac{1}{2}<a<1.
\end{align}
Then $f_a\in H^p$, $g_a\in VMOA$, and
$$
\sup\limits_{\frac{1}{2}<a<1}\|f_a\|_{H^p}\lesssim1\quad  and \quad \sup\limits_{\frac{1}{2}<a<1}\|g_a\|_{BMOA}\lesssim1.
$$
It is a remarkable fact that
$$
\frac{1}{(1-z)^\alpha}=\sum_{n=0}^{\infty}\frac{\Gamma(n+\alpha)}{\Gamma(n+1)\Gamma(\alpha)}z^n,
$$
and
$$
\frac{\Gamma(n+\alpha)}{\Gamma(n+1)\Gamma(\alpha)}\asymp n^{\alpha-1}
$$
by Stirling's formula. Then,
$$
R^{-1,\alpha-1}g_a(t)=1+\sum_{n=1}^{\infty}\frac{\Gamma(n+\alpha)}{\Gamma(n+1)\Gamma(\alpha)}\frac{(at)^n}{n}\asymp\frac{1}{(1-at)^{\alpha-1}}.
$$
Taking $r\in[a,1)$ and using (\ref{eqn-3.5}), we obtain that
$$
\begin{aligned}
1&\gtrsim\sup\limits_{\frac{1}{2}<a<1}\|f_a\|_{H^p}\sup\limits_{\frac{1}{2}<a<1}\|g_a\|_{BMOA}\\
&\gtrsim\left|\int_{[0,1)}^{}R^{-1,\alpha-1}g_a\left(r^2t\right)\overline{f_a(t)}d\mu(t)\right|\\
&\gtrsim\int_{[a,1)}^{}\frac{1}{(1-ar^2t)^{\alpha-1}}\frac{(1-a^2)^{\frac{1}{p}}}{(1-at)^{\frac{2}{p}}}d\mu(t)\\
&\gtrsim\frac{1}{(1-a^2)^{\frac{1}{p}+\alpha-1}}\mu([a,1)).
\end{aligned}
$$
This implies that $\mu$ is a $\left(\frac{1}{p}+\alpha-1\right)$-Carleson measure.

Conversely, if $\mu$ is a $\left(\frac{1}{p}+\alpha-1\right)$-Carleson measure. Using \cite[Lemma 3.2]{12}, we have that $\frac{d\mu(t)}{(1-t)^{\alpha-1}}$ is a $\frac{1}{p}$-Carleson measure. It is easy see that
$$
\int_{[0,1)}^{}\frac{|f(t)|}{(1-t)^{\alpha-1}}d\mu(t)\lesssim\|f\|_{H^p},\quad for\ all\ f\in H^p,0<p\le1.
$$
 Using this, Lemma \ref{Lm1.2} and $BMOA\subset\mathcal{B}$ (see \cite{4}), we obtain that
$$
\begin{aligned}
\left|\int_{[0,1)}^{}R^{-1,\alpha-1}g(r^2t)\overline{f(t)}d\mu(t)\right|&\lesssim\|g\|_\mathcal{B}\int_{[0,1)}^{}\frac{|f(t)|}{(1-r^2t)^{\alpha-1}}d\mu(t)\\
&\lesssim\|g\|_{BMOA}\int_{[0,1)}^{}\frac{|f(t)|}{(1-t)^{\alpha-1}}d\mu(t)\\
&\lesssim\|f\|_{H^p}\|g\|_{BMOA},\quad f\in H^p,g\in VMOA.
\end{aligned}
$$
Therefore, (\ref{eqn-3.5}) holds, and hence $\mathcal{H}_{\mu,\alpha}:H^p\rightarrow H^1$ is bounded.

(ii) If $\frac{1}{2}<\alpha<1$ and let $f_a(z)$ and $g_a(z)$ be the functions defined in (\ref{eqn-abc}), then
$$
R^{-1,\alpha-1}g_a(t)=1+\sum_{n=1}^{\infty}\frac{\Gamma(n+\alpha)}{\Gamma(n+1)\Gamma(\alpha)}\frac{(at)^n}{n}\asymp1.
$$
Taking $r\in[a,1)$ and using (\ref{eqn-3.5}), we obtain that
$$
\begin{aligned}
1&\gtrsim\sup\limits_{\frac{1}{2}<a<1}\|f_a\|_{H^p}\sup\limits_{\frac{1}{2}<a<1}\|g_a\|_{BMOA}\\
&\gtrsim\left|\int_{[0,1)}^{}R^{-1,\alpha-1}g_a(r^2t)\overline{f_a(t)}d\mu(t)\right|\\
&\gtrsim\int_{[a,1)}^{}\frac{(1-a^2)^{\frac{1}{p}}}{(1-at)^{\frac{2}{p}}}d\mu(t)\\
&\gtrsim\frac{1}{(1-a^2)^{\frac{1}{p}}}\mu([a,1)).
\end{aligned}
$$
This implies that $\mu$ is a $\frac{1}{p}$-Carleson measure.

If $\mu$ is a $\frac{1}{p}$-Carleson measure, then
$$
\int_{[0,1)}^{}|f(t)|d\mu(t)\lesssim\|f\|_{H^p},\quad for\ all\ f\in H^p,0<p\le1.
$$
Using Fubini's theorem and Lemma \ref{Lm1.3}, we obtain that
$$
\begin{aligned}
\|\mathcal{H}_{\mu,\alpha}(f)\|_{H^1}&=\sup\limits_{0<r<1}\frac{1}{2\pi}\int_{0}^{2\pi}\int_{[0,1)}^{}\frac{|f(t)|}{\left|1-tre^{i\theta}\right|^\alpha}d\mu(t)d\theta\\
&=\sup\limits_{0<r<1}\int_{[0,1)}^{}|f(t)|\left(\frac{1}{2\pi}\int_{0}^{2\pi}\frac{d\theta}{\left|1-tre^{i\theta}\right|^{\alpha}}\right)d\mu(t)\\
&\lesssim\int_{[0,1)}^{}|f(t)|d\mu(t)\lesssim\|f\|_{H^p}.
\end{aligned}
$$
Therefore, $\mathcal{H}_{\mu,\alpha}(H^p)\subset H^1$. The closed graph theorem implies that $\mathcal{H}_{\mu,\alpha}:H^p\rightarrow H^1$ is bounded.
\end{proof}


Theorem \ref{Th1.2} and \cite[Theorem 2]{7} together yield the following corollary.

\begin{corollary}
Let $\mu$ be a positive Borel measure on $[0,1),0<p\le1$. If $\mathcal{H}_{\mu,\alpha}:H^p\rightarrow H^1$ is bounded for some $\alpha>0$, then for any $0<{\alpha}'<\alpha, \mathcal{H}_{\mu,{\alpha}'}:H^p\rightarrow H^1$ is bounded.
\end{corollary}

\begin{corollary}
Suppose that $0<p\le1$ and $\alpha>0$. Let $\mu$ be a positive Borel measure on $[0,1)$.

(i) If $\alpha>1$ and $\int_{[0,1)}^{}\frac{d\mu(t)}{(1-t)^{\frac{1}{p}+\alpha-1}}<\infty$, then $\mathcal{H}_{\mu,\alpha}:H^p\rightarrow H^1$ is bounded.

(ii) If $0<\alpha<1$ and $\int_{[0,1)}^{}\frac{d\mu(t)}{(1-t)^\frac{1}{p}}<\infty$, then $\mathcal{H}_{\mu,\alpha}:H^p\rightarrow H^1$ is bounded.
\end{corollary}

\begin{proof}
(i) If $\alpha>1$ and $\int_{[0,1)}^{}\frac{d\mu(t)}{(1-t)^{\frac{1}{p}+\alpha-1}}<\infty$, using Lemma \ref{Lm1.2} and the fact that
\begin{align}\label{eqn-3.4}
|f(z)|\lesssim\frac{\|f\|_{H^p}}{(1-|z|)^\frac{1}{p}},\quad for \ all\ f\in H^p,z\in\mathbb{D}.
\end{align}
We have that
$$
\begin{aligned}
\left|\int_{[0,1)}^{}R^{-1,\alpha-1}g(r^2t)\overline{f(t)}d\mu(t)\right|&\lesssim\|f\|_{H^p}\|g\|_\mathcal{B}\int_{[0,1)}^{}\frac{1}{(1-r^2t)^{\alpha-1}(1-t)^\frac{1}{p}}d\mu(t)\\
&\lesssim\|f\|_{H^p}\|g\|_{BMOA}\int_{[0,1)}^{}\frac{1}{(1-t)^{\frac{1}{p}+\alpha-1}}d\mu(t)\\
&\lesssim\|f\|_{H^p}\|g\|_{BMOA},\quad f\in H^p,g\in VMOA.
\end{aligned}
$$
Therefore, (\ref{eqn-3.5}) holds, and hence $\mathcal{H}_{\mu,\alpha}:H^p\rightarrow H^1$ is bounded.

(ii) If $0<\alpha<1$ and $\int_{[0,1)}^{}\frac{d\mu(t)}{(1-t)^\frac{1}{p}}<\infty$, then (\ref{eqn-3.4}) shows that
$$
\begin{aligned}
\|\mathcal{H}_{\mu,\alpha}(f)\|_{H^1}&\le\sup\limits_{0<r<1}\frac{1}{2\pi}\int_{0}^{2\pi}\int_{[0,1)}^{}\frac{|f(t)|}{\left|1-tre^{i\theta}\right|^\alpha}d\mu(t)d\theta\\
&=\sup\limits_{0<r<1}\int_{[0,1)}^{}|f(t)|\frac{1}{2\pi}\int_{0}^{2\pi}\frac{d\theta}{\left|1-tre^{i\theta}\right|^{\alpha}}d\mu(t)\\
&\lesssim\int_{[0,1)}^{}|f(t)|d\mu(t)\\
&\lesssim\|f\|_{H^p}\int_{[0,1)}^{}\frac{d\mu(t)}{(1-t)^\frac{1}{p}}\lesssim\|f\|_{H^p}.
\end{aligned}
$$
This proof is finished.
\end{proof}

\begin{theorem}\label{th3.3}
Suppose that $0<p\le1, 1<q<\infty$ and $\alpha>0$. Let $\mu$ be a positive Borel measure on $[0,1)$ which satisfies the condition in Lemma \ref{Th1.1}.

(i) If $\alpha>\frac{1}{q}$, then $\mathcal{H}_{\mu,\alpha}:H^p\rightarrow H^q$ is bounded if and only if $\mu$ is a $\left(\frac{1}{p}+\frac{1}{{q}'}+\alpha-1\right)$-Carleson measure;

(ii) If $\alpha<\frac{1}{q}$, then $\mathcal{H}_{\mu,\alpha}:H^p\rightarrow H^q$ is bounded;

(iii) If $\alpha=\frac{1}{q}$ and $\mu$ is a $\frac{1}{q}$-logarithmic $\frac{1}{p}$-Carleson measure, then $\mathcal{H}_{\mu,\alpha}:H^p\rightarrow H^q$ is bounded.
\end{theorem}
\begin{proof}
(i) Recall the duality theorem \cite{3} for $H^q$ shows that $(H^q)^\ast\cong H^{{q}'}$ and $(H^{{q}'})^\ast\cong H^q (q>1)$, under the pairing
$$
<F,G>=\lim_{r\to1}\frac{1}{2\pi}\int_{0}^{2\pi}\overline{F\left(re^{i\theta}\right)}G\left(e^{i\theta}\right)d\theta,\quad F\in H^q,\ G\in H^{{q}'}.
$$
Therefore, it follows from (\ref{eqn3.2}) that $\mathcal{H}_{\mu,\alpha}:H^p\rightarrow H^q$ is bounded if and only if

\begin{align}\label{eqn-3.6}
\left|\int_{[0,1)}^{}R^{-1,\alpha-1}g(r^2t)\overline{f(t)}d\mu(t)\right|\lesssim\|f\|_{H^p}\|\|g\|_{H^{{q}'}},\quad f\in H^p,g\in {H^{{q}'}}.
\end{align}

Assume that $\mathcal{H}_{\mu,\alpha}:H^p\rightarrow H^q$ is bounded, take the families of text functions
$$
f_a(z)=\frac{(1-a^2)^{\frac{1}{p}}}{(1-az)^{\frac{2}{p}}},\quad g_a(z)=\frac{(1-a^2)^{\frac{1}{{q}'}}}{(1-az)^{\frac{2}{{q}'}}},\quad \frac{1}{2}<a<1.
$$
Then $f_a\in H^p$, $g_a\in H^{{q}'}$, and
$$
\sup\limits_{\frac{1}{2}<a<1}\|f_a\|_{H^p}\lesssim1\quad  and \quad \sup\limits_{\frac{1}{2}<a<1}\|g_a\|_{H^{{q}'}}\lesssim1.
$$
It is noteworthy that
$$
R^{-1,\alpha-1}g_a(t)=(1-a^2)^{\frac{1}{{q}'}}\sum_{n=0}^{\infty}\frac{\Gamma(n+\alpha)\Gamma(n+\frac{2}{{q}'})}{\Gamma(n+1)\Gamma(\alpha)\Gamma(n+1)\Gamma(\frac{2}{{q}'})}(at)^n.
$$
This implies that
$$
R^{-1,\alpha-1}g_a(t)\asymp\frac{(1-a^2)^{\frac{1}{{q}'}}}{(1-at)^{\frac{2}{{q}'}+\alpha-1}}
$$
by Stirling's formula. Taking $r\in[a,1)$ and using (\ref{eqn-3.6}), we have that
$$
\begin{aligned}
1&\gtrsim\sup\limits_{\frac{1}{2}<a<1}\|f_a\|_{H^p}\sup\limits_{\frac{1}{2}<a<1}\|g_a\|_{H^{{q}'}}\\
&\gtrsim\left|\int_{[0,1)}^{}R^{-1,\alpha-1}g_a(r^2t)\overline{f_a(t)}d\mu(t)\right|\\
&\gtrsim\int_{[a,1)}^{}\frac{(1-a^2)^{\frac{1}{{q}'}}}{(1-ar^2t)^{\frac{2}{{q}'}+\alpha-1}}\frac{(1-a^2)^{\frac{1}{p}}}{(1-at)^{\frac{2}{p}}}d\mu(t)\\
&\gtrsim\frac{1}{(1-a^2)^{\frac{1}{p}+\frac{1}{{q}'}+\alpha-1}}\mu([a,1)).
\end{aligned}
$$
This implies that $\mu$ is a $\left(\frac{1}{p}+\frac{1}{{q}'}+\alpha-1\right)$-Carleson measure.

Conversely, if $\mu$ is a $\left(\frac{1}{p}+\frac{1}{{q}'}+\alpha-1\right)$-Carleson measure. Using \cite[Lemma 3.2]{12}, we have that $\frac{d\mu(t)}{(1-t)^{\alpha-\frac{1}{q}}}$ is a $\frac{1}{p}$-Carleson measure. It is easy see that
$$
\int_{[0,1)}^{}\frac{|f(t)|}{(1-t)^{\alpha-\frac{1}{q}}}d\mu(t)\lesssim\|f\|_{H^p},\quad for\ all\ f\in H^p,0<p\le1,
$$
by (\ref{eqn3.1}). This together with Fubini's theorem and Lemma \ref{Lm1.3}, we have that
$$
\begin{aligned}
\left\|\mathcal{H}_{\mu,\alpha}(f)\right\|_{H^q}&\le \sup\limits_{0<r<1}\left\{\frac{1}{2\pi}\int_{0}^{2\pi}\left(\int_{[0,1)}^{}\frac{|f(t)|}{\left|1-tre^{i\theta}\right|^\alpha}d\mu(t)\right)^qd\theta\right\}^{\frac{1}{q}}\\
&=\sup\limits_{0<r<1}\int_{[0,1)}^{}|f(t)|\left(\frac{1}{2\pi}\int_{0}^{2\pi}\frac{d\theta}{\left|1-tre^{i\theta}\right|^{q\alpha}}\right)^{\frac{1}{q}}d\mu(t)\\
&\lesssim\int_{[0,1)}^{}\frac{|f(t)|}{(1-t)^{\alpha-\frac{1}{q}}}d\mu(t)\lesssim\|f\|_{H^p}.
\end{aligned}
$$
Therefore, $\mathcal{H}_{\mu,\alpha}(H^p)\subset H^q$. The closed graph theorem implies that $\mathcal{H}_{\mu,\alpha}:H^p\rightarrow H^q$ is bounded.

(ii) If $\alpha<\frac{1}{q}$, using Fubini's theorem, Lemma \ref{Lm1.3} and (\ref{eqn3.1}) we have that
$$
\begin{aligned}
\left\|\mathcal{H}_{\mu,\alpha}(f)\right\|_{H^q}&\le \sup\limits_{0<r<1}\left\{\frac{1}{2\pi}\int_{0}^{2\pi}\left(\int_{[0,1)}^{}\frac{|f(t)|}{\left|1-tre^{i\theta}\right|^\alpha}d\mu(t)\right)^qd\theta\right\}^{\frac{1}{q}}\\
&=\sup\limits_{0<r<1}\int_{[0,1)}^{}|f(t)|\left(\frac{1}{2\pi}\int_{0}^{2\pi}\frac{d\theta}{\left|1-tre^{i\theta}\right|^{q\alpha}}\right)^{\frac{1}{q}}d\mu(t)\\
&\lesssim\int_{[0,1)}^{}|f(t)|d\mu(t)\lesssim\|f\|_{H^p}.
\end{aligned}
$$
Therefore, $\mathcal{H}_{\mu,\alpha}:H^p\rightarrow H^q$ is bounded.

(iii) Since $\mu$ is a $\frac{1}{q}$-logarithmic $\frac{1}{p}$-Carleson measure. Using \cite[Proposition 2.5]{10}, we have that $\left ( \log\frac{e}{1-t} \right ) ^{\frac{1}{q}} d\mu(t)$ is a $\frac{1}{p}$-Carleson measure. It is easy see that
$$
\int_{[0,1)}^{}\frac{|f(t)|}{\left ( \log\frac{e}{1-t} \right ) ^{\frac{1}{q}}  }d\mu(t)\lesssim\|f\|_{H^p},\quad for\ all\ f\in H^p,0<p\le1.
$$
By Fubini's theorem and Lemma \ref{Lm1.3} we have that
$$
\begin{aligned}
\left\|\mathcal{H}_{\mu,\alpha}(f)\right\|_{H^q}&\le \sup\limits_{0<r<1}\left\{\frac{1}{2\pi}\int_{0}^{2\pi}\left(\int_{[0,1)}^{}\frac{|f(t)|}{\left|1-tre^{i\theta}\right|^\alpha}d\mu(t)\right)^qd\theta\right\}^{\frac{1}{q}}\\
&=\sup\limits_{0<r<1}\int_{[0,1)}^{}|f(t)|\left(\frac{1}{2\pi}\int_{0}^{2\pi}\frac{d\theta}{\left|1-tre^{i\theta}\right|^{q\alpha}}\right)^{\frac{1}{q}}d\mu(t)\\
&\lesssim\int_{[0,1)}^{}|f(t)|\left ( \log\frac{e}{1-t} \right ) ^{\frac{1}{q}} d\mu(t)\lesssim\|f\|_{H^p}.
\end{aligned}
$$
Therefore, $\mathcal{H}_{\mu,\alpha}:H^p\rightarrow H^q$ is bounded.
\end{proof}

\begin{theorem}\label{th3.7}
Suppose that $1<p\le q<\infty$  and $\alpha>1$. Let $\mu$ be a positive Borel measure on $[0,1)$ which satisfies the condition in Lemma \ref{Th1.1}. If $\int_{[0,1)}^{}\frac{1}{(1-t)^{\frac{1}{p}+\frac{1}{{q}'}+\alpha -1}}d\mu(t)<\infty$, then $\mathcal{H}_{\mu,\alpha}:H^p\rightarrow H^q$ is bounded.
\end{theorem}
\begin{proof}
%

Suppose that $\int_{[0,1)}^{}\frac{1}{(1-t)^{\frac{1}{p}+\frac{1}{{q}'}+\alpha -1}}d\mu(t)<\infty$. Setting $s=1+\frac{(\alpha-q)p}{q}$, then ${s}'=1+\frac{q}{(\alpha-q)p}$ is the conjugate exponent of $s$ and $\frac{1}{p}+\frac{1}{{q}'}+\alpha -1=\frac{1}{p}+\alpha-\frac{1}{q}=\frac{s}{p}=\left(\alpha-\frac{1}{q}\right){s}'$. Then, using \cite[Theorem 9.4]{4} we have
\begin{align}\label{eqn3.8}
\left(\int_{[0,1)}^{}|f(t)|^sd\mu(t)\right)^{\frac{1}{s}}\lesssim\|f\|_{H^p}, \quad for \ all \ f\in H^p.
\end{align}
Since $\alpha>\frac{1}{q}$, by Fubini's theorem, H$\ddot{o}$lder's inequality, Lemma \ref{Lm1.3} and (\ref{eqn3.8}) we have that
$$
\begin{aligned}
\left\|\mathcal{H}_{\mu,\alpha}(f)\right\|_{H^q}&\le \sup\limits_{0<r<1}\left\{\frac{1}{2\pi}\int_{0}^{2\pi}\left(\int_{[0,1)}^{}\frac{|f(t)|}{\left|1-tre^{i\theta}\right|^\alpha}d\mu(t)\right)^qd\theta\right\}^{\frac{1}{q}}\\
&=\sup\limits_{0<r<1}\int_{[0,1)}^{}|f(t)|\left(\frac{1}{2\pi}\int_{0}^{2\pi}\frac{d\theta}{\left|1-tre^{i\theta}\right|^{q\alpha}}\right)^{\frac{1}{q}}d\mu(t)\\
&\lesssim\int_{[0,1)}^{}\frac{|f(t)|}{(1-t)^{\alpha-\frac{1}{q}}}d\mu(t)\\
&\le\left(\int_{[0,1)}^{}|f(t)|^sd\mu(t)\right)^{\frac{1}{s}}\left(\int_{[0,1)}^{}\frac{1}{(1-t)^{\left(\alpha-\frac{1}{q}\right){s}'}}d\mu(t)\right)^{\frac{1}{{s}'}}\\
&\lesssim\|f\|_{H^p}.
\end{aligned}
$$
Therefore, $\mathcal{H}_{\mu,\alpha}:H^p\rightarrow H^q$ is bounded.
\end{proof}

\begin{theorem}\label{th-3.4}
For $1\le p\le2$ and $\alpha>1$, suppose that $\mu$ is a positive Borel measure on $[0,1)$ which satisfies the condition in Lemma \ref{Th1.1}. Then $\mathcal{H}_{\mu,\alpha}$ is a bounded operator in $H^p$ if and only if $\mu$ is an $\alpha$-Carleson measure.
\end{theorem}
\begin{proof}
If $p=1$, it follows from Theorem \ref{Th1.2} (i) that $\mathcal{H}_{\mu,\alpha}$ is a bounded operator in $H^1$ if and only if $\mu$ is an $\alpha$-Carleson measure.

If $p=2$, the proof of the necessity is analogous to Theorem \ref{Th1.2}. For the sufficiency, set $f(z)=\sum_{n=0}^{\infty}a_nz^n\in H^2$, then $\|f\|_{H^2}^2=\sum_{n=0}^{\infty}|a_n|^2$. Since $\mu$ is an $\alpha$-Carleson measure, we obtain
$$
|\mu_{n,k,\alpha}|=\frac{\Gamma(n+\alpha)}{\Gamma(n+1)\Gamma(\alpha)}|\mu_{n+k}|\lesssim n^{\alpha-1}\frac{1}{(n+k+1)^{\alpha}}.
$$
By this and Hilbert's inequality, we obtain that
$$
\begin{aligned}
\|\mathcal{H}_{\mu,\alpha}(f)\|_{H^2}^2&=\sum_{n=0}^{\infty}\left|\sum_{k=0}^{\infty}\mu_{n,k,\alpha}a_k\right|^2\\
&\le\sum_{n=0}^{\infty}\left(\sum_{k=0}^{\infty}\left|\mu_{n,k,\alpha}\right|\left|a_k\right|\right)^2\\
&\lesssim\sum_{n=0}^{\infty}n^{2(\alpha-1)}\left(\sum_{k=0}^{\infty}\frac{|a_k|}{(n+k+1)^{\alpha}}\right)^2\\
&\lesssim\sum_{k=0}^{\infty}|a_k|^2=\|f\|_{H^2}^2.
\end{aligned}
$$
It follows that  $\mathcal{H}_{\mu,\alpha}$ is a bounded operator in $H^2$. The complex interpolation theorem (see \cite[Theorem 2.4]{21}) implies that
$$
H^p=(H^2,H^1)_\theta,\quad  if \ 1<p<2\ and\ \theta=\frac{2}{p}-1.
$$
This shows that $\mathcal{H}_{\mu,\alpha}$ is a bounded operator in $H^p(1\le p\le2)$.
\end{proof}

\begin{remark}
 From \cite[Theorem 4.4]{17}, we know that, for $1\leq p\leq 2$, $\mathcal{H}_{\mu, 2}$ is a bounded operator on $H^p$ if and only if $\mu$ is a 2-Carleson measure. Theorem \ref{th-3.4} is a generation of Theorem 4.4 in \cite{17}.
\end{remark}
\section{Compactness of $\mathcal{H}_{\mu,\alpha }$ acting on $H^p$}
In this section, we characterize the compactness of the Generalized Hilbert operator $\mathcal{H}_{\mu,\alpha }$.

\begin{theorem}\label{Th3.5}  
For $0<p\le1$ and let $\mu$ be a positive Borel measure on $[0,1)$ which satisfies the condition in Lemma \ref{Th1.1}.

(i) If $\alpha>1$, then $\mathcal{H}_{\mu,\alpha}:H^p\rightarrow H^1$ is compact if and only if $\mu$ is a vanishing $\left(\frac{1}{p}+\alpha-1\right)$-Carleson measure;

(ii) If $0<\alpha<1$, then $\mathcal{H}_{\mu,\alpha}:H^p\rightarrow H^1$ is compact if and only if $\mu$ is a vanishing $\frac{1}{p}$-Carleson measure.
\end{theorem}
\begin{proof}
Suppose $0<p\le1\le q<\infty$, if $\mu$ is a $\frac{q}{p}$-Carleson measure, the identity mapping $i$ is well defined from $H^p$ into $L^q(\mathbb{D},\mu)$ and let the norm of $i$ denoted by $\mathcal{N}(\mu)$. For $0<s<1$, write
\begin{align}\label{eqn-mu1}
d\mu_s(z)=\chi_{s<|z|<1}(t)d\mu(t).
\end{align}
Thus, $\mu$ is a vanishing $\frac{q}{p}$-Carleson measure if and only if
\begin{align}\label{eqn-mu}
\mathcal{N}(\mu_s)\to0,\quad as\ s\to1^-.
\end{align}

(i) Let $\{a_n\}\subset[0,1)$ be any sequence with $a_n\to1$ and $f_{a_n}(z)$ be defined as in (\ref{eqn-abc}). Then $f_{a_n}\in H^p$, $\sup\limits_{n\ge1}\|f_{a_n}\|_{H^p}\asymp1$ and $\{f_{a_n}\}$ converges to 0 uniformly on any compact subset of $\mathbb{D}$. Since  $\mathcal{H}_{\mu,\alpha}:H^p\rightarrow H^1$ is compact, by Lemma \ref{lm2.5} we have that
$$
\lim_{n\to\infty}\|\mathcal{H}_{\mu,\alpha}(f_{a_n})\|_{H^1}=0.
$$
Then, together with (\ref{eqn-3.5}) we obtain that
$$
\lim_{n\to\infty}\left|\int_{[0,1)}^{}R^{-1,\alpha-1}g(r^2t)f_{a_n}(t)d\mu(t)\right|=0,\quad for\ all\ g\in VMOA.
$$
Take
$$
g_{a_n}(z)=\log\frac{e}{1-a_nz}.
$$
Thus,
$$
\begin{aligned}
\left|\int_{[0,1)}^{}R^{-1,\alpha-1}g_{a_n}(r^2t)\overline{f_{a_n}(t)}d\mu(t)\right|&\asymp\int_{[0,1)}^{}\frac{1}{(1-a_nr^2t)^{\alpha-1}}\frac{(1-a_n^2)^{\frac{1}{p}}}{(1-a_nt)^{\frac{2}{p}}}d\mu(t)\\
&\gtrsim(1-a_n^2)^{\frac{1}{p}}\int_{[a_n,1)}^{}\frac{1}{(1-a_nt)^{\frac{2}{p}+\alpha-1}}d\mu(t)\\
&\gtrsim\frac{\mu([a_n,1))}{(1-a_n)^{\frac{1}{p}+\alpha-1}}.
\end{aligned}
$$
Since $\{a_n\}$ is an arbitrary sequence on $[0,1)$, then
$$
\lim_{t\rightarrow1^-}\frac{\mu([t,1))}{(1-t)^{\frac{1}{p}+\alpha-1}}=0.
$$
Thus, $\mu$ is a vanishing $\left(\frac{1}{p}+\alpha-1\right)$-Carleson measure.

Suppose that $\mu$ is a vanishing $\left(\frac{1}{p}+\alpha-1\right)$-Carleson measure. Let $\{f_n\}_{n=1}^\infty$ be a bounded sequence of $H^p$ and $\lim_{n\to\infty}\{f_n\}=0$ on any compact subset of $\mathbb{D}$. By Lemma \ref{lm2.5}, it is suffice to prove that $\mathcal{H}_{\mu,\alpha}(f_n)\to0$ in $H^1$. For every $g\in VMOA,0<s<1$, we decude that
$$
\begin{aligned}
&\left|\int_{[0,1)}^{}R^{-1,\alpha-1}g(r^2t)\overline{f_n(t)}d\mu(t)\right|\\
\le&\left(\int_{[0,s]}^{}+\int_{(s,1)}^{}\right)\left|R^{-1,\alpha-1}g(r^2t)\overline{f_n(t)}d\mu(t)\right|\\
\le&\left(\int_{[0,s]}^{}+\int_{(s,1)}^{}\right)\left|R^{-1,\alpha-1}g(r^2t)\right|\left|f_n(t)\right|d\mu(t)\\
\end{aligned}
$$
Bearing in mind that $ \{f_n\}$ converges to 0 uniformly on every compact subset of $\mathbb{D}$, so we have
$$
\int_{[0,s]}^{}\left|R^{-1,\alpha-1}g(r^2t)\right|\left|f_n(t)\right|d\mu(t)\to0
$$
Since $\frac{d\mu(t)}{(1-t)^{\alpha-1}}$ is a vanishing $\frac{1}{p}$-Carleson measure by \cite[Lemma 3.2]{12}. Then
$$
\begin{aligned}
&\int_{(s,1)}^{}\left|R^{-1,\alpha-1}g(r^2t)\right|\left|f_n(t)\right|d\mu(t)\\
\lesssim&\ \|g\|_{\mathcal{B}}\int_{(s,1)}^{}|f_n(t)|\frac{d\mu(t)}{(1-r^2t)^{\alpha-1}}\\
\lesssim&\ \|g\|_{BMOA}\int_{[0,1)}^{}|f_n(t)|\frac{d\mu_s(t)}{(1-t)^{\alpha-1}}\\
\lesssim&\ \mathcal{N}(\mu_s)\|g\|_{BMOA}\|f_n(t)\|_{H^p}.
\end{aligned}
$$
Then, using (\ref{eqn-mu}), this also tends to $0$. Therefore, we obtain that
$$
\begin{aligned}
\lim_{n\to\infty}\int_{[0,1)}^{}\left|f_n(t)\right|\left|R^{-1,\alpha-1}g(r^2t)d\mu(t)\right|=0,\quad for\ all \ g\in VMOA.
\end{aligned}
$$
Therefore,
$$
\begin{aligned}
\lim_{n\to\infty}\left|\int_{0}^{2\pi}\overline{\mathcal{H}_{\mu,\alpha}(f_n)(re^{i\theta})}g(re^{i\theta})d\theta\right|=0,\quad for\ all \ g\in VMOA.
\end{aligned}
$$
Thus, $\mathcal{H}_{\mu,\alpha}(f_n)\rightarrow0$ in $H^1$.

(ii) If $0<\alpha<1$, then
$$
R^{-1,\alpha-1}g_a(t)=1+\sum_{n=1}^{\infty}\frac{\Gamma(n+\alpha)}{\Gamma(n+1)\Gamma(\alpha)}\frac{(at)^n}{n}\asymp1.
$$
Arguing as in the proof of (i), we will obtain the necessity.

If $\mu$ is a vanishing $\frac{1}{p}$-Carleson measure and let $\{f_n\}_{n=1}^\infty$ be a bounded sequence of $H^p$ and $\lim_{n\to\infty}\{f_n\}=0$ on any compact subset of $\mathbb{D}$. By Lemma \ref{lm2.5}, it is suffice to prove that $\mathcal{H}_{\mu,\alpha}(f_n)\to0$ in $H^1$. Arguing as in the proof of the boundedness in Theorem \ref{Th1.2} (ii), it implies that
$$
\begin{aligned}
\|\mathcal{H}_{\mu,\alpha}(f_n)\|_{H^1}&\lesssim\int_{[0,1)}^{}|f_n(t)|d\mu(t)\\
&\lesssim\int_{[0,s]}^{}|f_n(t)|d\mu(t)+\int_{(s,1)}^{}|f_n(t)|d\mu(t)\\
&\lesssim\int_{[0,s]}^{}|f_n(t)|d\mu(t)+\int_{[0,1)}^{}|f_n(t)|d\mu_s(t)\\
&\lesssim\int_{[0,s]}^{}|f_n(t)|d\mu(t)+\mathcal{N}(\mu_s)\|f_n(t)\|_{H^p},\quad g\in VMOA.
\end{aligned}
$$
Then, using (\ref{eqn-mu}) and the fact that $\{f_n\}\to0$ uniformly on any compact subset of $\mathbb{D}$, we obtain that this tends to 0. Thus,
$$
\lim_{n\to\infty}\|\mathcal{H}_{\mu,\alpha}(f_n)\|_{H^1}=0
$$
By Lemma \ref{lm2.5}, we conclude that $\mathcal{H}_{\mu,\alpha}:H^p\rightarrow H^1$ is compact.
\end{proof}
\begin{theorem}
For $0<p\le1,1<q<\infty,\alpha>0$, and let $\mu$ be a positive Borel measure on $[0,1)$ which satisfies the condition in Lemma \ref{Th1.1}.

(i) If $\alpha>\frac{1}{q}$, then $\mathcal{H}_{\mu,\alpha}:H^p\rightarrow H^q$ is compact if and only if $\mu$ is a vanishing $\left(\frac{1}{p}+\frac{1}{{q}'}+\alpha-1\right)$-Carleson measure;

(ii) If $\alpha=\frac{1}{q}$ and $\mu$ is a vanishing $\frac{1}{q}$-logarithmic $\frac{1}{p}$-Carleson measure, then $\mathcal{H}_{\mu,\alpha}:H^p\rightarrow H^q$ is compact.
\end{theorem}
\begin{proof}
(i) and (ii) can be proved similarly to the proof of Theorem \ref{Th3.5}, so we omit the details.
\end{proof}

\begin{theorem}
If $1\le p\le2,\alpha>1$ and let $\mu$ be a positive Borel measure on $[0,1)$ which satisfies the condition in Lemma \ref{Th1.1}. Then $\mathcal{H}_{\mu,\alpha}$ is a compact operator in $H^p$ if and only if $\mu$ is a vanishing $\alpha$-Carleson measure.
\end{theorem}

\begin{proof}
If $p=1$, it follows from Theorem \ref{Th3.5} (i) that $\mathcal{H}_{\mu,\alpha}$ is a bounded operator in $H^1$ if and only if $\mu$ is a vanishing $\alpha$-Carleson measure.

If $p=2$, the proof of the necessity is analogous to Theorem \ref{Th3.5}.

Suppose that $\mu$ is a vanishing $\alpha$-Carleson measure. Let $\{f_s\}_{n=1}^\infty$ be a bounded sequence of $H^2$ and $\lim_{n\to\infty}\{f_s\}=0$ on any compact subset of $\mathbb{D}$. By Lemma \ref{lm2.5}, it is suffice to prove that $\mathcal{H}_{\mu,\alpha}(f_{s})\to0$ in $H^2$. Since $\mu$ is a vanishing $\alpha$-Carleson measure, $\mu_{n,k}=o(\frac{1}{(n+k+1)^\alpha})$ as $n\to\infty$. If
$$
\mu_{n,k}=\frac{\varepsilon_n}{(n+k+1)^\alpha},\quad n=0,1,2,...,
$$
then $\{\varepsilon_n\}\to0$. If, for every $s$,
$$
f_{s}(z)=\sum_{k=0}^{\infty}a_k^{(s)}z^k,\quad z\in\mathbb{D}.
$$
By this and Hilbert's inequality, we obtain
$$
\begin{aligned}
\sum_{n=0}^{\infty}\left|\sum_{k=0}^{\infty}\frac{a_k^{(s)}}{(n+k+1)^{\alpha}}\right|^2\le\pi^2\sum_{k=0}^{\infty}|a_k^{(s)}|^2\le\pi^2.
\end{aligned}
$$
Take $\varepsilon>0$ and then take $N$ such that
$$
n\ge N\quad \Rightarrow\quad \varepsilon_n^2<\frac{\varepsilon}{2\pi^2}.
$$
Then,
$$
\begin{aligned}
\|\mathcal{H}_{\mu,\alpha}(f_{s})\|_{H^2}^2&=\sum_{n=0}^{\infty}\left|\sum_{k=0}^{\infty}\mu_{n,k,\alpha}a_k^{(s)}\right|^2\\
&=\sum_{n=0}^{\infty}n^{2(\alpha-1)}\sum_{k=0}^{\infty}\left|\mu_{n,k}a_k^{(s)}\right|^2\\
&\le\sum_{n=0}^{N}n^{2(\alpha-1)}\left|\sum_{k=0}^{\infty}\mu_{n,k}a_k^{(s)}\right|^2+\sum_{n=N+1}^{\infty}n^{2(\alpha-1)}\left|\sum_{k=0}^{\infty}\mu_{n,k}a_k^{(s)}\right|^2\\
&\lesssim\sum_{n=0}^{N}n^{2(\alpha-1)}\left|\sum_{k=0}^{\infty}\mu_{n,k}a_k^{(s)}\right|^2+\sum_{n=0}^{\infty}n^{2(\alpha-1)}\left|\sum_{k=0}^{\infty}\frac{\varepsilon_na_k^{(s)}}{(n+k+1)^{\alpha}}\right|^2\\
&\le\sum_{n=0}^{N}n^{2(\alpha-1)}\left|\sum_{k=0}^{\infty}\mu_{n,k}a_k^{(s)}\right|^2+\frac{\varepsilon}{2\pi^2}\sum_{n=0}^{\infty}\left|\sum_{k=0}^{\infty}\frac{a_k^{(s)}}{n+k+1}\right|^2\\
&\le\sum_{n=0}^{N}n^{2(\alpha-1)}\left|\sum_{k=0}^{\infty}\mu_{n,k}a_k^{(s)}\right|^2+\frac{\varepsilon}{2}.
\end{aligned}
$$
Then, the fact that $\{f_{s}\}\to0$ uniformly on any compact subset of $\mathbb{D}$, we obtain that
$$
\sum_{n=0}^{N}n^{2(\alpha-1)}\left|\sum_{k=0}^{\infty}\mu_{n,k}a_k^{(s)}\right|^2\to0,\quad as\ s\to\infty.
$$
Then it follows that there exist $s_0\in N$ such that $\|\mathcal{H}_{\mu,\alpha}(f_{s})\|_{H^2}^2<\varepsilon$ for all $j\ge j_0$.\\
Thus,
$$
\lim_{s\to\infty}\|\mathcal{H}_{\mu,\alpha}(f_{s})\|_{H^2}=0.
$$
By Lemma \ref{lm2.5}, we conclude that $\mathcal{H}_{\mu,\alpha}$ is a compact operator in $H^2$.
The complex interpolation theorem implies that
$$
H^p=(H^2,H^1)_\theta,\quad  if \ 1<p<2\ and\ \theta=\frac{2}{p}-1.
$$
Since $H^2$ is reflexive, and $\mathcal{H}_{\mu,\alpha}$ is compact on $H^1$ and $H^2$, using \cite[Theorem 10]{20} shows that $\mathcal{H}_{\mu,\alpha}$ is a compact operator in $H^p(1\le p\le2)$.
\end{proof}

We recall that an operator $S$ on a separable Hilbert space $Y$ is a Hilbert-Schmidt operator if
$$
\sum_{k=0}^{\infty}\|S(e_k)\|_{Y}^2<\infty
$$
for an orthonormal basis $\{e_k\}_{k=0}^\infty$ of $Y$. The finiteness of this sum does not depend on the basis chosen. In \cite{9}, The measure for which $\mathcal{H}_{\mu}$ is a Hilbert-Schmidt operator on $H^2$ has been characterized. As a matter of fact, we will be able to obtain sufficient and necessary condition which $\mathcal{H}_{\mu,\alpha}$ is a Hilbert-Schmidt operator on $H^2$.

\begin{theorem}
For $\alpha>0$ and let $\mu$ be a positive Borel measure on $[0,1)$ which satisfies the condition in Lemma \ref{Th1.1}. Then $\mathcal{H}_{\mu,\alpha}$ is a Hilbert-Schmidt operator on $H^2$ if and only if
\begin{align}\label{Hil-Sch}
\int_{[0,1)}^{}\frac{\mu([t,1))}{(1-t)^{2\alpha}}d\mu(t)<\infty.
\end{align}
\end{theorem}
\begin{proof}
Take the orthonormal basis $\{e_k\}_{k=0}^\infty=z^k$ and notice that
$$
\begin{aligned}
\sum_{k=0}^{\infty}\|\mathcal{H}_{\mu,\alpha}(e_k)\|_{H^2}^2&=\sum_{k=0}^{\infty}\sum_{n=0}^{\infty}|\mu_{n,k,\alpha}|^2\\
&=\sum_{k=0}^{\infty}\sum_{n=0}^{\infty}n^{2(\alpha-1)}|\mu_{n,k}|^2\\
&=\sum_{k=0}^{\infty}\sum_{n=0}^{\infty}n^{2(\alpha-1)}\int_{[0,1)}^{}\int_{[0,1)}^{}(ts)^{n+k}d\mu(s)d\mu(t)\\
&\asymp\int_{[0,1)}^{}\frac{\mu([t,1))}{(1-t)^{2\alpha}}d\mu(t).
\end{aligned}
$$
Therefore, the operator  $\mathcal{H}_{\mu,\alpha}$ is a Hilbert-Schmidt operator on $H^2$ if and only if (\ref{Hil-Sch}) holds.
\end{proof}

\section{Essential norm of $\mathcal{H}_{\mu,\alpha}$ on $H^p$ }
In this section, we will give the essential norm of the operator $\mathcal{H}_{\mu,\alpha}$ from  $H^p(0<p\le1)$ into $H^p(1\le q<\infty)$.

\begin{theorem}
For $0<p\le1,\alpha>1$, and let $\mu$ be a $\left(\frac{1}{p}+\alpha-1\right)$-Carleson measure on $[0,1)$.  Then $$\|\mathcal{H}_{\mu,\alpha}\|_{e,H^p\to H^1}\approx\limsup_{s\to1^-}\frac{\mu([s,1))}{(1-s)^{\frac{1}{p}+\alpha-1}}.$$
\end{theorem}

\begin{proof}
We now give the upper estimate of $\mathcal{H}_{\mu,\alpha}$ from $H^p(0<p\le1)$ to $H^1$.

Since $\mu$ is a $\left(\frac{1}{p}+\alpha-1\right)$-Carleson measure on $[0,1)$, the operator $\mathcal{H}_{\mu,\alpha}$ from $H^p(0<p\le1)$ to $H^1$ is bounded by Theorem \ref{Th1.2}. For any $0<s<1$, let the positive measure $\mu_s$ defined by (\ref{eqn-mu1}). It is straightforward to confirm that $\mu_s$ is a vanishing $\left(\frac{1}{p}+\alpha-1\right)$-Carleson measure. We conclude that $\mathcal{H}_{\mu_s,\alpha}$ is compact from $H^p(0<p\le1)$ to $H^1$ by  Theorem 4.1. Then
$$\|\mathcal{H}_{\mu,\alpha}\|_{e,H^p\to H^1}\le\left\|\mathcal{H}_{\mu,\alpha}-\mathcal{H}_{\mu_s,\alpha}\right\|_{H^p\to H^1}=\inf\limits_{\|f\|_{H^p}=1}\left\|\mathcal{H}_{{\mu-\mu_s},\alpha}(f)\right\|_{H^1}.$$
By (\ref{eqn3.2}) we obtain that
$$
\begin{aligned}
&\left|\int_{0}^{2\pi}\overline{\mathcal{H}_{{\mu-\mu_s},\alpha}(f)(re^{i\theta})}g(re^{i\theta})d\theta\right|=\left|\int_{[0,1)}^{}R^{-1,\alpha-1}g(r^2t)\overline{f(t)}d(\mu-\mu_s)(t)\right|\\
\lesssim&\|g\|_\mathcal{B}\int_{[0,1)}^{}\frac{|f(t)|}{(1-r^2t)^{\alpha-1}}d(\mu-\mu_s)(t)\lesssim\|g\|_{BMOA}\int_{[0,1)}^{}\frac{|f(t)|}{(1-t)^{\alpha-1}}d(\mu-\mu_s)(t)\\
\lesssim&\|g\|_{BMOA}\|f\|_{H^p}\|\nu-\nu_s\|.
\end{aligned}
$$
where $d\nu(t)=\frac{1}{(1-t)^{\alpha-1}}d\mu(t)$ and $d\nu_s(t)=\frac{1}{(1-t)^{\alpha-1}}d\mu_s(t)$. By \cite[Lemma 3.2]{12}, we know that the positive measure $\nu-\nu_s$ is a $\frac{1}{p}$-Carleson measure. Thus,
$$\|\mathcal{H}_{\mu,\alpha}\|_{e,H^p\to H^1}\lesssim\limsup_{s\to1^-}\frac{\mu([s,1))}{(1-s)^{\frac{1}{p}+\alpha-1}}.$$

Now we  give the lower estimate for $\mathcal{H}_{\mu,\alpha}$.

Let $0<\tau<1$ and
$$f_\tau(z)=\frac{(1-\tau^2)^\frac{1}{p}}{(1-\tau z)^\frac{2}{p}}=\sum_{k=0}^{\infty}a_{k,\tau}z^n,$$
where $a_{k,\tau}=O\left((1-\tau^2)^{\frac{1}{p}}k^{\frac{2}{p}-1}\tau^k\right)$. Then $\{f_\tau\}$ is a bounded sequence in $H^p$ and $\lim_{\tau\to1^-}f_\tau(z)=0$ on any compact subset of $\mathbb{D}$. Since $f_\tau\to0$ weakly in $H^p$, we get that $\|Kf_\tau\|\to0$ as $\tau\to1$ for any compact operator $K$ on $H^p$. Furthermore
$$
\left\|\mathcal{H}_{\mu,\alpha}-K\right\|_{H^p\to H^1}\ge \left\|\left(\mathcal{H}_{\mu,\alpha}-K\right)f_\tau\right\|_{H^1}\ge\left\|\mathcal{H}_{\mu,\alpha}f_\tau\right\|_{H^1}-\left\|Kf_\tau\right\|_{H^1}.
$$
Using Lemma \ref{lm2.6}, we derive that
$$
\begin{aligned}
&\left\|\mathcal{H}_{\mu,\alpha}(f_\tau)\right\|_{H^1}\ge\sup\limits_{n}\sum_{k=0}^{\infty}\mu_{n,k,\alpha}a_{k,\tau}\\
=&\sup\limits_{n}\frac{\Gamma(n+\alpha)}{\Gamma(n+1)\Gamma(\alpha)}(1-\tau^2)^{\frac{1}{p}}\sum_{k=0}^{\infty}k^{\frac{2}{p}-1}\tau^k\int_{[0,1)}^{}t^{n+k}d\mu(t)\\
\ge&\sup\limits_{n}n^{\alpha-1}(1-\tau^2)^{\frac{1}{p}}\sum_{k=0}^{\infty}k^{\frac{2}{p}-1}\tau^k\int_{[s,1)}^{}t^{n+k}d\mu(t)\\
\ge&\sup\limits_{n}n^{\alpha-1}(1-\tau^2)^{\frac{1}{p}}s^n\mu([s,1))\sum_{k=0}^{\infty}k^{\frac{2}{p}-1}\tau^{k}s^k\\
=&\sup\limits_{n}n^{\alpha-1}s^n\frac{(1-\tau^2)^{\frac{1}{p}}}{{(1-s\tau)^{\frac{2}{p}}}}\mu([s,1)).
\end{aligned}
$$
Let $\tau=s$ and we choose $n$ such that $1-\frac{1}{n+1}\le s<1-\frac{1}{n}$. We find that
$$
\left\|\mathcal{H}_{\mu,\alpha}(f_\tau)\right\|_{H^1}\ge\sup\limits_{n}\frac{1}{e(1-s^2)^{\frac{1}{p}}(1-s)^{\alpha-1}}\mu([s,1))\ge\sup\limits_{n}\frac{1}{e(1-s)^{\frac{1}{p}+\alpha-1}}\mu([s,1)).
$$
It follows that
$$
\|\mathcal{H}_{\mu,\alpha}\|_{e,H^p\to H^1}\ge\limsup_{\tau\to1^-}\|\mathcal{H}_{\mu,\alpha}f_\tau\|_{H^1}
\gtrsim\limsup_{s\to1^-}\frac{\mu([s,1))}{(1-s)^{\frac{1}{p}+\alpha-1}}
$$
The proof is finished.
\end{proof}

\begin{theorem}
For $0<p\le1,1<q<\infty,\alpha>1$, and let $\mu$ be a $\left(\frac{1}{p}+\frac{1}{{q}'}+\alpha-1\right)$-Carleson measure on $[0,1)$. Then
$$
\|\mathcal{H}_{\mu,\alpha}\|_{e,H^p\to H^q}\approx\limsup_{s\to1^-}\frac{\mu([s,1))}{(1-s)^{\frac{1}{p}+\frac{1}{{q}'}+\alpha-1}}.
$$
\end{theorem}

\begin{proof}
Since $\mu$ is a $\left(\frac{1}{p}+\frac{1}{{q}'}+\alpha-1\right)$-Carleson measure on $[0,1)$, the operator $\mathcal{H}_{\mu,\alpha}$ from $H^p$ to $H^q$ is bounded by Theorem \ref{th3.3}. For any $0<s<1$, let the positive measure $\mu_s$ defined by (\ref{eqn-mu1}). It is straightforward to confirm that $\mu_s$ is a vanishing $\left(\frac{1}{p}+\frac{1}{{q}'}+\alpha-1\right)$-Carleson measure. We conclude that $\mathcal{H}_{\mu_s,\alpha}$ is compact from $H^p$ to $H^q$ by Theorem 4.1. Then
$$\|\mathcal{H}_{\mu,\alpha}\|_{e,H^p\to H^q}\le\left\|\mathcal{H}_{\mu,\alpha}-\mathcal{H}_{\mu_s,\alpha}\right\|_{H^p\to H^q}=\inf\limits_{\|f\|_{H^p}=1}\left\|\mathcal{H}_{{\mu-\mu_s},\alpha}(f)\right\|_{H^q}.$$
Appealing to Lemma \ref{Lm1.3}, we obtain that
$$
\begin{aligned}
\left\|\mathcal{H}_{\mu-\mu_s,\alpha}(f)\right\|_{H^q}&\le \sup\limits_{0<r<1}\left\{\frac{1}{2\pi}\int_{0}^{2\pi}\left(\int_{[0,1)}^{}\frac{|f(t)|}{\left|1-tre^{i\theta}\right|^\alpha}d\left(\mu-\mu_s\right)(t)\right)^qd\theta\right\}^{\frac{1}{q}}\\
&\le\sup\limits_{0<r<1}\int_{[0,1)}^{}|f(t)|\left(\frac{1}{2\pi}\int_{0}^{2\pi}\frac{d\theta}{\left|1-tre^{i\theta}\right|^{q\alpha}}\right)^{\frac{1}{q}}d\left(\mu-\mu_s\right)(t)\\
&\lesssim\int_{[0,1)}^{}\frac{|f(t)|}{(1-t)^{\alpha-\frac{1}{q}}}d\left(\mu-\mu_s\right)(t)\\
&\lesssim\|f\|_{H^p}\|\nu-\nu_s\|.
\end{aligned}
$$
where $d\nu(t)=\frac{1}{(1-t)^{\alpha-\frac{1}{q}}}d\mu(t)$ and $d\nu_s(t)=\frac{1}{(1-t)^{\alpha-\frac{1}{q}}}d\mu_s(t)$. The positive measure $\nu-\nu_s$ is a $\frac{1}{p}$-Carleson measure by \cite[Lemma 3.2]{12}. Thus,
$$\|\mathcal{H}_{\mu,\alpha}\|_{e,H^p\to H^q}\lesssim\limsup_{s\to1^-}\frac{\mu([s,1))}{(1-s)^{\frac{1}{p}+\frac{1}{{q}'}+\alpha-1}}.$$

On the other hand, let $0<\tau<1$ and
$$f_\tau(z)=\frac{(1-\tau^2)^\frac{1}{p}}{(1-\tau z)^\frac{2}{p}}.$$
Then $\{f_\tau\}$ is a bounded sequence in $H^p$ and $\lim_{\tau\to1^-}f_\tau(z)=0$ on any compact subset of $\mathbb{D}$. Since $f_\tau\to0$ weakly in $H^p$, we get that $\|Kf_\tau\|\to0$ as $\tau\to1$ for any compact operator $K$ on $H^p$. Furthermore
$$
\left\|\mathcal{H}_{\mu,\alpha}-K\right\|_{H^p\to H^q}\ge \left\|\left(\mathcal{H}_{\mu,\alpha}-K\right)f_\tau\right\|_{H^q}\ge\left\|\mathcal{H}_{\mu,\alpha}f_\tau\right\|_{H^q}-\left\|Kf_\tau\right\|_{H^q}.
$$
by Fej$\acute{e}$r-Riesz inequality (see \cite[page 46]{3}) we have that

$$
\begin{aligned}
\left\|\mathcal{H}_{\mu,\alpha}(f_\tau)\right\|_{H^q}&=\left(\frac{1}{2\pi}\int_{0}^{2\pi}\left|\int_{0}^{1}\frac{f_\tau(t)}{\left(1-te^{i\theta}\right)^\alpha}d\mu(t)\right|^qd\theta\right)^{\frac{1}{q}}\\
&\gtrsim\left(\int_{0}^{1}\left|\int_{0}^{1}\frac{f_\tau(t)}{\left(1-tx\right)^\alpha}d\mu(t)\right|^qdx\right)^{\frac{1}{q}}\\
&=\left(\int_{0}^{1}\left(\int_{0}^{1}\frac{(1-\tau^2)^{\frac{1}{p}}}{\left(1-\tau t\right)^{\frac{2}{p}}(1-tx)^\alpha }d\mu(t)\right)^qdx\right)^{\frac{1}{q}}\\
&\ge\left(\int_{0}^{1}\left(\int_{\tau}^{1}\frac{(1-\tau^2)^{\frac{1}{p}}}{\left(1-\tau t\right)^{\frac{2}{p}}(1-tx)^\alpha }d\mu(t)\right)^qdx\right)^{\frac{1}{q}}\\
&\gtrsim\frac{\mu([\tau,1))}{(1-\tau)^{\frac{1}{p}}}\left(\int_{0}^{1}\frac{1}{(1-\tau x)^{q\alpha}}dx\right)^{\frac{1}{q}}\\
&\asymp\frac{\mu([\tau,1))}{(1-\tau)^{\frac{1}{p}+\alpha-\frac{1}{q}}}.
\end{aligned}
$$
It follows that
$$
\|\mathcal{H}_{\mu,\alpha}\|_{e,H^p\to H^q}\ge\limsup_{\tau\to1^-}\|\mathcal{H}_{\mu,\alpha}f_\tau\|_{H^q}
\gtrsim\limsup_{s\to1^-}\frac{\mu([\tau,1))}{(1-\tau)^{\frac{1}{p}+\frac{1}{{q}'}+\alpha-1}}
$$
The proof is finished.
\end{proof}

\subsection*{Conflicts of Interest}
The authors declare that there are no conflicts of interest regarding the publication of this paper.

 \end{document}